\documentclass[reqno]{amsart}
\usepackage{amsmath, amssymb, amsthm, epsfig,latexsym}

\usepackage{color}

 \newtheorem{theorem}{Theorem}%[section]
 \newtheorem{lemma}[theorem]{Lemma}
 \newtheorem{proposition}[theorem]{Proposition}
 
 \theoremstyle{definition}
 \newtheorem{definition}[theorem]{Definition}
 
 \theoremstyle{remark}
 
 \newcommand{\mc}{\mathcal}
 
 \newcommand{\C}{\mathbb{C}}
 \newcommand{\R}{\mathbb{R}}
 \newcommand{\N}{\mathbb{N}}
 
 \newcommand{\Z}{\mathbb{Z}}
 \newcommand{\csch}{\mathrm{csch}}

 \def\today{\ifcase\month\or
  January\or February\or March\or April\or May\or June\or
  July\or August\or September\or October\or November\or December\fi
  \space\number\day, \number\year}

\title{Extremal functions with vanishing condition}
\author{Friedrich Littmann and Mark Spanier}

\date{\today}
\address{Department of mathematics, North Dakota State University, Fargo, ND 58105.}
\email{Friedrich.Littmann@ndsu.edu, Mark.Spanier@ndsu.edu}

\begin{document}

\begin{abstract} For a Hermite-Biehler function $E$ of mean type $\tau$ we determine the optimal (with respect to the de Branges measure of $E$) majorant $M_E^+$ and minorant $M_E^-$ of exponential type $\tau$ for the truncation of $x\mapsto (x^2+a^2)^{-1}$.  We prove that
\[
\int_\R\left(M_E^+(x) - M_E^-(x)\right) |E(x)|^{-2}dx = \frac{1}{a^2 K(0,0)}
\]
where $K$ is the reproducing kernel for the de Branges space $\mc{H}(E)$. As an application we determine the optimal majorant and minorant for the Heaviside function that vanish at a fixed point $\alpha = ia$ on the imaginary axis. We show that the difference of majorant and minorant has integral value $(\pi a - \tanh(\pi a))^{-1} \pi a$.
\end{abstract}

\maketitle

\numberwithin{equation}{section}

\section{Results}

Let $\delta\ge 0$.  We say that an entire function $F$ is of exponential type $\delta\ge 0$ if for every $\varepsilon>0$ there exists $C_\varepsilon>0$ such that $|F(z)| \le C_\varepsilon e^{(\delta+\varepsilon)|z|}$ for all $z\in\C$, and we write  $\mc{A}(\delta)$ for the class of entire functions of exponential type $\delta$. 

Let $\mu$ be a Borel measure on $\R$.  We write $\mc{A}_p(\delta,\mu)$ for the space of functions $F\in \mc{A}(\delta)$  such that
\[
\int_\R |F(x)|^p d\mu(x)<\infty, 
\]
and we write $\mc{A}_p(\delta)$ if $\mu$ is Lebesgue measure.

Let $f:\R\to \R$ be measurable and of polynomial growth, and let $\alpha\in\C\backslash\R$. Consider the class $\mc{F}_+ = \mc{F}_+(f,\alpha)\subseteq \mc{A}(\delta)$ of entire functions $F$ satisfying
\begin{itemize}
\item[(i)] $F(x) \ge f(x)$ for all real $x$,
\item[(ii)] $F(\alpha) =0$,
\item[(iii)] $||F-f||_{L^1(\R)}<\infty$.
\end{itemize}

Assume that $\mc{F}_+\neq \emptyset$. We seek to find $F^+\in \mc{F}_+$ such that the inequality
\begin{align}\label{ex-i-ii}
\int_\R (F(x) - F^+(x)) dx \ge 0
\end{align}
holds for all $F\in \mc{F}_+$.   If $F^+$ exists, we call it an \textsl{extremal majorant of $f$ (of type $\delta$) with vanishing at $\alpha$}. The class $\mc{F}_-(f,\alpha)$ is defined analogously by reversing the inequalities in (i) and \eqref{ex-i-ii}; the corresponding optimal functions are called  \textsl{extremal minorants with vanishing at $\alpha$}. We note that condition (ii) implies $F(\overline{\alpha}) =0$.

This extremal problem was first considered in \cite{KV} for the function
\[
f(x) = \chi_{[\beta,\gamma]}(x)
\]
with $\beta<\gamma$. M.\ Kelly found non-optimal functions in $\mc{F}_+(\chi_{[\beta,\gamma]},\alpha)$ and obtained bounds for $||F^\pm-\chi_{[\beta,\gamma]}||_1$ as a function of $\delta$. This choice of $f$  is quite natural, since the extremal functions satisfying (i) and \eqref{ex-i-ii} have found frequent applications in analytic number theory \cite{Gal, S, V} and signal processing \cite{DL}.

The condition (ii) may allow  applications to certain explicit formulas for $L$-functions. We briefly describe the setup. An $L$-function $s\mapsto L(\chi,s)$ is given by
\[
L(\chi,s)  = \sum_{n=1}^\infty \frac{\chi(n)}{n^s}
\]
where $\chi$ is a Dirichlet character. The generalized Riemann hypothesis states that the zeros of $L$ in $0<\Re s<1$ all lie on the line $\Re s=1/2$. An explicit formula is an identity that relates a series involving values $L(\chi,\rho) h(\rho)$ where $L(\chi,\rho) =0$ and $h$ is a smooth function to values of the Fourier transform of $h$. Under assumption of GRH and writing $f(t)  = L(\chi,\frac12 + it)$ the series involves only real values $t$, and $h$ is frequently chosen to be a solution to the above extremal problem. Without assumption of GRH the series will involve non-real values of $t$, and an extremal function satisfying (ii) may be used to eliminate the contribution of such a zero.

The results in this paper are an attempt to understand how the vanishing condition affects the construction of extremal functions. For this purpose we consider  $\alpha=ia$ with $a>0$ and approximate the cutoff function $x_+^0$ defined by
\[
x_+^0 = \begin{cases}
1 \text{ if }x\ge 0,\\
0\text{ else.}
\end{cases}
 \]
 
 We find the extremal majorant and minorant  $F^\pm \in \mc{F}_\pm(x_+^0, ia)$ of type $2\pi$. The results, given in Theorem \ref{vanishing} below, are obtained by changing measure and function in such a way that condition (ii) may be dropped.  This is possible since $F^+$ and $F^-$ are necessarily of the form 
\[
F^\pm(z) = G^\pm(z) (z^2+a^2)
 \] 
for some entire functions $G^\pm\in\mc{A}(\delta)$. Hence, defining $t_a$ by 
\[
t_a(x) = \begin{cases}
(x^2+a^2)^{-1}\text{ if }x\ge 0,\\
0\text{ else},
\end{cases}
\]
we seek to find $G^+,G^-\in \mc{A}(\delta)$ such that $G^-(x) \le t_a(x) \le G^+(x)$ for all real $x$, and the integral
\[
\int_{-\infty}^\infty \big(G^+(x) - G^-(x)\big) (x^2+a^2) dx
\]
is minimal.  This type of problem has been frequently investigated. An early result is the solution by Beurling \cite{Be} and independently by Selberg \cite[Chapter 20]{S} for the signum function, cf.\ the account in Vaaler \cite{V}.  Further references to recent developments can be found in \cite{CLV}. For the related subject of one-sided polynomial approximation to the signum function see Bustamante {\it et.al.}\ \cite{BQM}, while the best one-sided approximations to characteristic functions by trigonometric polynomials were found recently by Babenko {\it et.al.}\ \cite{BKY}.  For the best approximation without constraints we refer to Ganzburg \cite{G}. For weighted polynomial approximation  see Lubinsky \cite{Lub-Poisson, Lub}. For general facts regarding best approximation we refer to the books of Korneichuk \cite{Ko} and Pinkus \cite{P}.

 In Theorem \ref{T1} we solve the extremal problem for $t_a$ with respect to a large class of measures. In order to state this result we require some terminology from the theory of de Branges Hilbert spaces of entire functions. We refer to \cite{B} and \cite{HV} for proofs of the following facts.

For any entire function $E$ we use the notation $E^*(z) = \overline{E(\overline{z})}$. An entire function $E$ satisfying the inequality
\begin{align}\label{HB}
|E(z)|>|E^*(z)|
\end{align}
for all $z$ with $\Im z >0$ will be called a \textsl{Hermite-Biehler} (HB) function. An analytic function $F$ defined in the upper half plane $\C^+$ is said to have bounded type, if $F$ is the quotient of two bounded functions. The number $v(F)$ defined by
\[
v(F) = \limsup_{y\to \infty} y^{-1} \log|F(iy)|
\]
is called the mean type of $F$. If $F$ has bounded type in $\C^+$, then its mean type is finite.

%We define a Borel measure $\mu_E$ by
%\begin{align}\label{DBS-measure}
%\mu_E(M) = \int_M \frac{dx}{|E(x)|^{2}}
%\end{align}
%for all Borel subsets $M$ of $\R$.
We denote by $\mc{H}(E)$  the vector space of entire functions $F$ such that
\[
\int_{-\infty}^\infty |F(x)/E(x)|^2 dx <\infty
\]
and the functions $F/E$ and $F^*/E$ have bounded type and nonpositive mean type in $\C^+$. It is a Hilbert space with scalar product
\[
\langle F, G \rangle_E = \int_{-\infty}^\infty F(x) G^*(x) |E(x)|^{-2} dx.
\]

 We define $A = (1/2)(E+E^*)$ and $B = (i/2)(E - E^*)$. A fundamental result of de Branges from the 1960's is the recognition that this space is a reproducing kernel Hilbert space; we briefly sketch the argument. It follows from \eqref{HB} that $E$ has no zeros in the open upper half plane, and it follows from the definition of $\mc{H}(E)$ that $F/E$ and $F^*/E$ have no zeros in an open set containing the closed upper half plane. The condition that $F/E$ and $F^*/E$ have non-positive mean type implies that the Cauchy integral formula for the upper half plane holds for $F/E$ and $F^*/E$ (e.g., \cite[Theorem 12]{B}, note also that Cauchy formulas for $F/E^*$ and $F^*/E^*$ in the lower half plane hold), and it follows with an elementary calculation that $\mc{H}(E)$ has the reproducing kernel $K$ given by
\begin{align}\label{K-from-AB}
K(w,z) = \frac{B(z) A(\bar{w}) - A(z) B(\bar{w})}{\pi(z-\bar{w})}
\end{align}
for $z\neq \bar{w}$.

The prototypical example of a de Branges space is the classical Paley-Wiener space of square integrable functions of finite exponential type at most $\tau$. In this case $E(z) = e^{-i\tau z}$, and the condition that $F/E$ and $F^*/E$ have non-positive mean type implies that $F$ and $F^*$ grow no faster than $E$, i.e., this condition is essentially equivalent to the statement that $F$ has exponential type $\tau$. The reproducing kernel for this space is the familiar sinc function
\[
(z,w)\mapsto \frac{\sin\tau (z-\bar{w})}{\pi(z-\bar{w})}.
\]

The classical proof that this is a reproducing kernel uses Fourier analysis, but this can also be shown using a contour integration argument. It is the latter proof that generalizes to the setting of de Branges spaces.

We require the following conditions for $E$:
\begin{enumerate}
\item[(I)] $E$ is of bounded type in $\C^+$ with mean type $\tau$,
\item[(II)] $E$ has no  real zeros, 
\item[(III)] $E^*(z) = E(-z)$ for all $z$,
\item[(IV)] $B \notin \mc{H}(E)$.
%\begin{align}\label{B-condition}
%\int_{-\infty}^\infty | E(x) - E^*(x)|^2 =\infty.
%\end{align}
\end{enumerate}

These are technical assumptions; we return to these properties in Section \ref{DBS}, where we obtain the following result.

\begin{theorem}\label{T1} Let $E$ be a HB-function satisfying conditions {\rm (I) -- (IV)}. If  $F^+,F^-\in\mc{A}(2 \tau)$ with  
\begin{align}\label{pmii}
F^-(x) \le t_a(x)\le F^+(x)
\end{align}
 for all real $x$, then
\begin{align}\label{opt-value-intro}
\int_\R (F^+(x) - F^-(x))  \frac{dx}{|E(x)|^2} \ge \frac{1}{a^2 K(0,0)},
\end{align}
and there exist functions $T^{\pm}_a\in\mc{A}(2 \tau)$ satisfying \eqref{pmii} such that there is equality in \eqref{opt-value-intro} for $F^+ = T^+_a$ and $F^- = T_a^-$.
\end{theorem}

We consider now the problem of finding the extremal functions with vanishing at $\alpha= ia$ with respect to Lebesgue measure. In order to apply Theorem \ref{T1} we require a Hermite-Biehler function $E_a$ satisfying (I) - (IV) such that 
\[
\int_{-\infty}^\infty \left| \frac{F(x)}{E_a(x)} \right|^2 dx = \int_{-\infty}^\infty |F(x)|^2(x^2+a^2) dx
\]
for all $F\in\mc{A}_2(\pi,(x^2+a^2) dx)$. We shall prove in Theorem \ref{isometry} that $E_a$ defined by
\begin{align}\label{def-Ea-intro}
E_a(z) = \left(\frac{2}{\sinh(2\pi a)} \right)^{\frac12} \frac{\sin\pi(z+ia)}{z+ia}.
\end{align}
has this property. Theorem \ref{T1} becomes applicable and we obtain the following result for extremal functions with vanishing at $\alpha = ia$.

\begin{theorem}\label{vanishing} Let $a>0$. If $S,T\in \mc{A}(2\pi)$ with 
\[
S(i a) = T(i a) =0
\]
and
\[
S(x) \le x_+^0 \le T(x)
\]
for all real $x$, then
\begin{align}\label{optimal-vanishing-intro}
\int_{-\infty}^\infty (S(x) - T(x)) dx \ge\frac{\pi a}{\pi a - \tanh(\pi a)},
\end{align}
and there exist $S_a^+, S_a^-\in\mc{A}(2\pi)$ such that there is equality in \eqref{optimal-vanishing-intro} for $S = S_a^-$ and $T = S_a^+$.
\end{theorem}

It is instructive to consider the corresponding integral for the extremal functions of type $2\pi \delta$. To find these values, we  temporarily set $a = b\delta$. Then the extremal functions $F^\pm \in \mc{A}(2\pi)$ with $F(ib\delta) =0$ satisfy
\[
\int_{-\infty}^\infty (F^+(x) - F^-(x)) dx = \frac{\pi b \delta}{\pi b\delta - \tanh(\pi b \delta)}.
\]

 We  note that the functions $F^\pm_\delta$ given by $ F^\pm_\delta(z) =  F^\pm(\delta z)$ are the extremal functions of type $2\pi \delta$ with $F^\pm_\delta(ib) =0$, and a change of variable gives 
 \[
 \int_{-\infty}^\infty ( F^+_\delta(x) - F^-_\delta(x)) dx =   \frac{\pi b }{\pi b\delta - \tanh(\pi b \delta)}.
 \]

 This implies that for fixed $b$ and $\delta\to \infty$ the integral is $\sim \delta^{-1}$, while for $\delta \to 0+$ the integral is $\sim 3 (\pi b)^{-2}  \delta^{-3}$. In \cite[Theorem 8]{V} the corresponding extremal problem for the signum function without the vanishing condition is treated. The integral value is shown to be equal to $\delta^{-1}$. This shows that the prescribed vanishing at $\alpha = ib$ substantially affects the integral value for small values of $\delta$, but the vanishing condition leads only to a small change if $\delta$  becomes large.

This paper is structured as follows. In Section \ref{interpolation} we give a general method to construct entire functions $M^\pm$ that interpolate $t_a$ at the zeros of a given Laguerre-P\'{o}lya entire function $F$, and we prove that 
\[
x\mapsto F(x) \left(M^\pm(x) - t_a(x)\right)
\] 
is of one sign for all real $x$. These functions serve as  candidates for the extremal functions in the approximation results of the subsequent sections. We prove Theorem \ref{T1} in Section \ref{DBS}. In Section \ref{section-vanishing-space} we investigate the Hermite-Biehler function $E_a$ from \eqref{def-Ea-intro} and prove Theorem \ref{vanishing}.

\medskip
 
\noindent {\it Acknowledgement.} The authors thank the anonymous referees for their valuable suggestions.

 %%%
%%%           Interpolation at zeros of Laguerre-Polya entire functions
%%%

\section{Interpolation at zeros of LP-functions}\label{interpolation}

\smallskip

For certain discrete subsets $\mc{T}\subset \R$ we show how to construct entire functions $F\in \mc{A}(2\tau)$ which interpolate $t_a$ at the elements of $\mc{T}$ (with prescribed multiplicity) \textsl{so that $F - t_a$ has no sign changes between consecutive elements of $\mc{T}$.} This is the basis for the construction of the extremal functions in Theorems \ref{T1} and \ref{vanishing}.

The interpolation theorems of this section are based on the representation
\[
\frac{a}{z^2+a^2} = \int_0^\infty e^{-z\lambda} \sin( a \lambda) d\lambda,
\]
valid for all $z$ with $\Re z>0$. The analogues of the crucial inequalities in Theorems \ref{M^+thm} and \ref{M^-thm} below for the function $x\mapsto e^{-x \lambda}\chi_{[0,\infty)}(x)$  are known for every $\lambda>0$  (cf.\ \cite{CL}), but the fact that $\sin(a\lambda) d\lambda$ is a signed measure introduces additional difficulties.

\begin{definition}
The Laguerre-P\'{o}lya class $LP$ consists of all entire functions of the form
\begin{equation}
F(z) = Ce^{-\gamma z^2+bz}z^\kappa \prod_{k=1}^\infty \left(1-\frac{z}{a_k}\right)e^{z/a_k}
\end{equation}
where $\gamma\geq 0,$ $\kappa \in \mathbb{N}_0,$ $b,$ $C,$ $a_k$ ($k\in \N$) are real, and 
\begin{equation}
\sum_{k=1}^\infty \frac{1}{a_k^2}<\infty.
\end{equation}

We denote by $\mc{T}_F$ the set of zeros of $F$. If the set of zeros is bounded above, we include $\infty\in\mc{T}_F$, if the set is bounded below, we include $-\infty\in\mc{T}_F$. 
\end{definition}

For $c\in \mathbb{R}\setminus \mathcal{T}_F$  we define
\begin{equation}\label{g_cDef}
g_c(t)=\frac{1}{2\pi i}\int_{c-i\infty}^{c+i \infty} \frac{e^{st}}{F(s)}\,ds,
\end{equation}
if this integral converges absolutely. (This is the case if $\gamma>0$ or if $F$ has at least two zeros.) Let $\tau_1$ and $\tau_2$ be the consecutive elements from $\mc{T}_F$ that satisfy $\tau_1<c<\tau_2$. A Fourier inversion shows that
\begin{equation}\label{OneOverF}
\frac{1}{F(z)}=\int_{-\infty}^\infty e^{-zt}g_c(t)\,dt
\end{equation}
in the strip $\tau_1 < \Re{z} <\tau_2$. An application of the residue theorem shows that $g_c = g_d$ for $c,d\in(\tau_1,\tau_2)$. 

As part of a series of papers on total positivity, I.J.\ Schoenberg \cite{Sch} gave an intrinsic characterization of the functions $g$ that may occur as Laplace inverse transformations of LP functions. We require estimates that can be found in \cite{HW}.

\begin{lemma}\label{gc-sign-changes} Let $F\in LP$  have at least $n+2$ zeros counted with multiplicity, and let $\tau_1,\tau_2\in \mc{T}_F$ with $0\in(\tau_1,\tau_2)$. Then $g_0^{(k)}$ exists for $k\le n$ and has at most $k$ sign changes on the real line. 
\end{lemma}

\begin{proof} This is \cite[Chapter IV, Theorem 5.1]{HW}.
\end{proof}

\begin{lemma}\label{g-growth} Let $F\in LP$ have at least two zeros. Let $\tau_1<\tau_2$ be two consecutive elements in $\mc{T}_F$, and let $c\in(\tau_1,\tau_2)$. Then there exist polynomials $P_n$ and $Q_n$ such that
\begin{align}\label{g-growth-estimate}
\begin{split}
|g_c^{(n)}(t)|\le |P_n(t)| e^{\tau_1 t}\text{ as }t\to\infty,\\
|g_c^{(n)}(t)|\le |Q_n(t)| e^{\tau_2 t}\text{ as }t\to -\infty.
\end{split}
\end{align}
\end{lemma}

\begin{proof} Let $t>0$ . The estimate follows from \eqref{g_cDef} by moving the integration path to $\Re s= d$ with $d<\tau_1$ using the residue theorem, see, e.g., \cite[Chapter V, Theorem 2.1]{HW}. For $t<0$ the integration path is moved to $\Re s = d$ with $d>\tau_2$.
\end{proof}

\begin{lemma}\label{g''properties} Let $F$ be an even $LP$-function with a double zero at the origin and at least one positive zero $\tau$. Assume that $F$ is positive in $(0,\tau)$, and define $g = g_{\tau/2}$ by \eqref{g_cDef}. 
\begin{enumerate}
\item The derivatives $g'$ and $g''$ exist and are nonnegative on the real line. 
\item The function $g''$ is even.
 \item If $-\tau< \Im \zeta<\tau$, then
\begin{equation}\label{g''CI}
\int_{-\infty}^\infty g''(\lambda)\cos(\zeta \lambda)\,d\lambda = -\frac{\zeta^2}{F(i\zeta)}.
\end{equation}
In particular, $F$ is real-valued on the imaginary axis.
\item If $-\tau<\Im \zeta<\tau$ and $w\in\R$, then
\begin{equation}\label{g''SI}
\int_{-\infty}^\infty g''(\lambda+w)\sin(\zeta \lambda)\,d\lambda = \frac{\zeta^2 \sin(\zeta w)}{F(i\zeta)}.
\end{equation}\end{enumerate}
\end{lemma}

\begin{proof} Since $F$ has at least four zeros counted with multiplicity, \eqref{g_cDef} may be differentiated twice under the integral sign, which shows that $g'$ and $g''$ exist.  Two integration by parts show that
\begin{align}\label{z^2overF}
\frac{z^j}{F(z)} = \int_{-\infty}^\infty e^{-zt} g^{(j)}(t) dt\qquad (j\in\{0,1,2\})
\end{align}
for all $z$ with $0<\Re z< \tau$. Since $z^{-j} F(z)$ is in $LP$ for $j\in \{0,1,2\}$, Lemma \ref{gc-sign-changes} with $k=0$ implies that $g^{(j)}$ has no sign changes on the real line. Evaluation of \eqref{z^2overF} at $z = \tau/2$ shows that these derivatives are nonnegative on the real line. Since $z^{-2} F(z)$ has no zero at the origin, \eqref{z^2overF} extends to $-\tau<\Re z<\tau$ for $j=2$. Since $z^{-2} F(z)$ is an even function of $z$, \eqref{g_cDef} with $c = 0$ gives
\[
g''(t) = \frac{1}{2\pi} \int_{-\infty}^\infty \frac{(iu)^2 e^{iut}}{F(iu)} du
\]
which implies that $g''$ is even. 

 Let $\zeta$ such that $-\tau<\Im \zeta<\tau$. Since \eqref{z^2overF} extends to $-\tau<\Re z<\tau$, we have
\[
\int_{-\infty}^\infty g''(\lambda)\cos(\zeta \lambda)\,d\lambda  = \frac12 \left(\frac{(i\zeta)^2}{F(i\zeta)}+ \frac{(-i\zeta)^2}{F(-i\zeta)}\right),
\]
and \eqref{g''CI} follows since $F$ is even. Since  $\lambda \mapsto g''(\lambda) \sin\zeta\lambda$ is odd, the trigonometric identity $\sin\zeta(\lambda - w) = \cos\zeta w\sin\zeta \lambda - \cos\zeta \lambda \sin\zeta w$ implies
\[
\int_{-\infty}^\infty g''(\lambda) \sin\zeta(\lambda-w) d\lambda = -\sin\zeta w \int_{-\infty}^\infty g''(\lambda) \cos\zeta\lambda d\lambda.
\]

An application of \eqref{g''CI} gives \eqref{g''SI}.
\end{proof}

Define for positive $a$ the function $h_a:\R\to\R$ by
\begin{align}\label{h-def}
h_a(w) = -\int_{-\infty}^0 g(\lambda+w) \sin a\lambda d\lambda, 
\end{align}
and note that \eqref{g-growth-estimate} with $\tau_2 = \tau$ implies that $g(t)$ and $g'(t)$ decay exponentially as $t\to -\infty$. Hence $h_a(w)$ is finite for every $w\in\R$, and an integration by parts gives for all real $w$
\begin{align}\label{h-rep}
h_a(w) = \frac{1}{a} \int_{-\infty}^0 g'(\lambda+w) (1-\cos a\lambda) d\lambda.
\end{align}

We require bounds for the derivatives of $h_a$ and evaluations for special values.

\begin{lemma}\label{hAlphaPropsLemma} Let $a >0$. Let $F\in LP$ be even with a double zero at the origin, at least one positive zero $\tau$, and at least five zeros (counted with multiplicity). Assume that $F$ is positive in $(0,\tau)$, and define $g = g_{\tau/2}$ by \eqref{g_cDef}. 
\begin{enumerate}
\item The inequalities
\begin{align}\label{hAlphaDerivativeBounds}
0\leq h_a^{(n)}(w) \leq \frac{2}{a}g^{(n)}(w)
\end{align}
hold for  $n \in \{0,1\}$ and all real $w$, and for $n=2$ and $w \leq 0$.
\item We have the representation
\begin{align}\label{ha(0)}
h_a'(0) = \frac{g'(0)}{a}+\frac{a}{2 F(ia)}.
\end{align}
\item For all real $w$ we have
\begin{align}\label{differentOfh_alpha''}
h_a''(w)- h_a''(-w)=-\sin(a w)\frac{a^2}{F(ia)}.
\end{align}
\end{enumerate}
\end{lemma}

\begin{proof}   Equation \eqref{h-rep} implies for all $n$ and all real $w$
\begin{align*}
h_a^{(n)}(w) = \int_{-\infty}^0 g^{(n+1)}(\lambda+w)\frac{1-\cos(a \lambda)}{a}\, d\lambda.
\end{align*}

The functions $g$, $g'$, and $g''$ are nonnegative on $\R$ by Lemma \ref{g''properties}, and $g'''$ has exactly one change of  sign on $\R$ by Lemma \ref{gc-sign-changes} applied to $g''$. Since $g''$ is even, the sign change is located at the origin.  It follows that for all real $w$ and $n\in\{0,1\}$, as well as for $n=2$ and $w\le 0$,
\begin{align*}
0\leq h_a^{(n)}(w) \leq \int_{-\infty}^0 g^{(n+1)}(\lambda + w) \frac{2}{a}\,d\lambda = \frac{2}{a}g^{(n)}(w),
\end{align*}
which implies inequality \eqref{hAlphaDerivativeBounds}. 

To prove \eqref{ha(0)} we differentiate \eqref{h-def} and set $w=0$ to get
\begin{align*}
h_a'(0)=-\int_{-\infty}^0 g'(\lambda) \sin(a \lambda)\,d\lambda.
\end{align*}

We perform an integration by parts, apply that $g''$ is even, and use \eqref{g''CI} to obtain
\begin{align*}
h_a'(0)&=\frac{g'(0)}{a}-\frac{1}{a}\int_{-\infty}^0g''(\lambda)\cos(a \lambda)\,d\lambda\\
			&=\frac{g'(0)}{a}-\frac{1}{2 a}\int_{-\infty}^\infty g''(\lambda)\cos(a \lambda)\,d\lambda\\
			&= \frac{g'(0)}{a}+\frac{a}{2 F(ia)}
\end{align*}
which finishes the proof of \eqref{ha(0)}.  Equations \eqref{h-def} and \eqref{g''SI} give
\begin{align*}
h_a''(w) 
&= -\int_{-\infty}^\infty g''(\lambda+w)\sin(a \lambda)\,d\lambda+\int_{0}^\infty g''(\lambda+w)\sin(a \lambda)\,d\lambda\\
&=-\sin(a w)\frac{a^2}{F(ia)}+h_a''(-w),
\end{align*}
which is \eqref{differentOfh_alpha''}.
\end{proof}

The next proposition investigates two interpolations of $t_a$ in $\Re z \ge 0$ and $\Re z \le \tau$, respectively, and shows that they are representations of a single entire function $z\mapsto A(F,a,z)$ which interpolates $x\mapsto t_a(x-)$ at the zeros of $F$. See \cite{CL} for a similar construction for the cutoff of an exponential function. 

\begin{proposition}\label{continuationProp} Let $a >0$. Let $F\in LP$ be even with a double zero at the origin and at least one positive zero $\tau$. Assume that $F$ is positive in $(0,\tau)$, and define $g = g_{\tau/2}$ by \eqref{g_cDef} and $h_a$ by \eqref{h-def}.
Define
\begin{align}\label{A-def}
\begin{split}
A_1(F,a,z)&= \frac{F(z)}{a}\int_{-\infty}^0 h_a(w) e^{-zw}\,dw \text{ for } \Re{z}<\tau\\
A_2(F,a,z)&= \frac{1}{z^2+a^2} - \frac{F(z)}{a}\int_{0}^\infty h_a(w)e^{-zw}\,dw \text{ for } \Re{z}>0.
\end{split}
\end{align}
Then $z \mapsto A_1(F,a,z)$ is analytic in $\Re z < \tau$, $z \mapsto A_2(F,a,z)$ is analytic in $\Re z > 0$, and these functions are restrictions of an entire function $z\mapsto A(F,a,z)$. 
Moreover, there exists a constant $c>0$ so that
\begin{align}\label{boundForA}
|A(F,a,z)|\leq c(1+|F(z)|)
\end{align}
for all $z\in \mathbb{C}$.
\end{proposition}

\begin{proof} Inequality \eqref{g-growth-estimate} with $\tau_2 = \tau$ implies that $g'(t)$ decays exponentially as $t\to -\infty$. Hence it follows from \eqref{h-rep} that
\begin{align}\label{h-estimate}
0\le h_a(w) \le \frac2a g(w)
\end{align}
for all real $w$, and \eqref{g-growth-estimate} with $\tau_2=\tau$ applied to $g$ for $t\to -\infty$ implies that the integral defining $A_1(F,a,z)$ converges absolutely in $\Re z<\tau$. Inequality \eqref{g-growth-estimate}  implies with $\tau_1 =0$ that $g$ has polynomial growth on the positive real axis, hence the integral in the definition of $A_2(F,a,z)$ converges absolutely for $\Re z>0$. It follows that $A_1$ and $A_2$ are analytic functions in $\Re z>\tau$ and $0<\Re z$, respectively.

To prove that $A_1$ and $A_2$ are analytic continuations of each other it suffices to prove that they are equal in the strip $0<\Re z<\tau$. Starting point is the identity
\begin{align}\label{ta-rep}
\frac{a}{z^2+a^2} = \int_0^\infty e^{-z\lambda} \sin a\lambda d\lambda,
\end{align}
valid for $\Re z>0$. Combining this with \eqref{OneOverF} gives for $0<\Re z<\tau$
\begin{align*}
\frac{1}{z^2+a^2} &= \frac{F(z)}{a} \int_{-\infty}^\infty \int_{0}^\infty e^{-z(w+\lambda)}g(w)\sin(a \lambda)\,d\lambda\,dw\\
						&=\frac{F(z)}{a} \int_{-\infty}^\infty \int_{0}^\infty e^{-z w}g(w-\lambda)\sin(a \lambda)\,d\lambda\,dw\\
                       &= \frac{F(z)}{a} \int_{-\infty}^\infty h_a(w)e^{-zw}\,dw.
\end{align*}

Inserting this in \eqref{A-def} shows that $A_1(F,a,z) = A_2(F,a,z)$ for $0<\Re z<\tau$. To prove \eqref{boundForA} we note that inequality \eqref{h-estimate} implies in $\Re z\le  \tau/2$
\begin{equation}
|A(F,a,z)|\leq \frac{|F(z)|}{a}\int_{-\infty}^0 |h_a(w) e^{-zw}|\,dw\leq \frac{2|F(z)|}{a^2}\int_{-\infty}^0 g(w) e^{-\tau w/2}\,dw,
\end{equation}
and an analogous calculation holds in $\Re z\ge \tau/2$.
\end{proof}

Starting with the function $A(F,a,z)$, we construct interpolations $M^\pm$ of $t_a$ that interpolate $t_a$ at the zeros of $F$ (with correct multiplicity) so that $M^\pm - t_a$ has no sign changes between two consecutive zeros of $F$. This is accomplished by selecting the value at the origin appropriately. We assume that $a>0$, and that $F\in LP$ and $\tau>0$ satisfy the assumptions of Proposition \ref{continuationProp}. We define $z\mapsto M^+(F,a,z)$ and $z\mapsto M^-(F,a,z)$ by 
\begin{align}
M^-(F,a,z)&=A(F,a,z) + \frac{h_a(0)}{a}\frac{F(z)}{z},\label{mDef}\\
\label{M^+Def} M^+(F,a,z)&= A(F,a,z) + \frac{h_a(0)}{a}\frac{F(z)}{z}+\frac{2g'(0)}{a^2}\frac{F(z)}{z^2}
\end{align}
where $A(F,a,z)$ is defined in \eqref{A-def}. Evidently $M^+$ and $M^-$ are entire functions. Recall that $\mc{T}_F\cap \R$ is the zero set of $F$. It is evident from the definitions that  $M^\pm(F,a,\xi) = t_a(\xi)$  for all real $\xi \in \mc{T}_F\backslash\{0\}$. Since $F$ has a double zero at the origin, we see that $M^-(F,a,0) =0$. 

  Since $g''$ is nonnegative and integrable on $\R$, \eqref{z^2overF} implies
\[
\frac{2}{F''(0)} = \int_{-\infty}^\infty g''(w)\,dw.
\]
As $g''$ is even and $g'(w)$  decays exponentially as $w\to -\infty$, we also have that $\int_{-\infty}^\infty g''(w)\,dw = 2g'(0)$. Hence, $z^{-2} F(z) \to 1/(2g'(0))$ as $z\to 0$ and $M^+(F,a,0) = 1/a^2$. This means that
\begin{align}\label{interpolationPoints}
M^\pm(F,a,\xi) = t_a(\xi\pm)
\end{align}
for all real $\xi \in \mc{T}_F$, where $t_a(\xi\pm)$ denotes the one sided limits at $\xi$.
 
 \begin{theorem}\label{M^+thm}  Let $a >0$. Let $F\in LP$ be even with a double zero at the origin and at least one positive zero $\tau$. Assume that $F$ is positive in $(0,\tau)$, and define $g = g_{\tau/2}$ by \eqref{g_cDef} and $h_a$ by \eqref{h-def}. If 
\[
F(ia) <0, 
\]
then
\begin{align}\label{M^+Sign}
F(x)\left\{M^+(F,a,x) - t_a(x)\right\}\geq 0
\end{align}
 for all real $x$.
 \end{theorem}

\begin{proof}  Consider first $x<0$. An expansion of the second term in \eqref{M^+Def} in a Laplace transform together with \eqref{A-def} gives
\begin{align}\label{M+ta-eq1}
M^+(F,a,x)-t_a(x) = \frac{F(x)}{a}\int_{-\infty}^0 (h_a(w)-h_a(0))e^{-x w}\,dw +\frac{F(x)}{x^2}\frac{2g'(0)}{a^2}.
\end{align}

Two integration by parts and \eqref{ha(0)} lead to
\begin{align*}
\int_{-\infty}^0 (h_a(w)-h_a(0))e^{-x w}\,dw &= \frac{1}{x^2} \int_{-\infty}^0 h_a''(w) e^{-wx} dx - \frac{h_a'(0)}{x^2}\\
&=\frac{1}{x^2} \int_{-\infty}^0 h_a''(w)e^{-x w}\,dw -\frac{g'(0)}{ax^2}-\frac{a}{2 x^2 F(ia)},
\end{align*}
and inserting this in \eqref{M+ta-eq1} gives
\begin{align}\label{M+ta-eq2}
M^+(F,a,x)-t_a(x)=\frac{F(x)}{x^2}\left( \frac{g'(0)}{a^2}-\frac{1}{2 F(ia)} +\frac{1}{a} \int_{-\infty}^0 h_a''(w)e^{-x w}\,dw  \right).
\end{align}

By assumption $-F(ia)>0$, and \eqref{hAlphaDerivativeBounds} implies $h_a''(w)\ge 0$. Since by assumption $z^{-1} F(z)$ is a LP-function that is positive in $(0,\tau)$, it follows from Lemma \ref{g''properties} that $g'(0)>0$. Hence \eqref{M^+Sign} is shown for $x<0$.

Let $x>0$. From \eqref{A-def}, \eqref{h-def}, and \eqref{ha(0)} we get
\begin{align*}%\label{M-ta-eq1}
\begin{split}
M^+(F,a,x)-t_a(x) &= -\frac{F(x)}{a}\int_0^\infty h_a(w)e^{-x w}\,dw +\frac{F(x)}{a}\frac{h_a(0)}{x}+\frac{F(x)}{x^2}\frac{2g'(0)}{a^2}\\
&= -\frac{F(x)}{a}\int_0^\infty (h_a(w)-h_a(0))e^{-x w}\,dw+\frac{F(x)}{x^2}\frac{2g'(0)}{a^2},
\end{split}
\end{align*}
and, analogously to \eqref{M+ta-eq2}, we obtain the representation
\begin{align}\label{M+ta-eq3}
M^+(F,a,x)-t_a(x)=\frac{F(x)}{x^2}\left(\frac{g'(0)}{a^2}-\frac{1}{2 F(ia)}-\frac{1}{a}\int_0^\infty h_a''(w)e^{-x w}\,dw \right)
\end{align}
for $x>0$. In order to investigate the sign of the right hand side, we multiply \eqref{differentOfh_alpha''} by $e^{-xw}$ and integrate $w$ over $[0,\infty)$ to get with \eqref{ta-rep} 
\begin{align*}
\int_0^\infty (h_a''(w) - h_a''(-w)) e^{-xw} dw &= -\frac{a^2}{F(ia)} \int_0^\infty e^{-xw}\sin(a w) dw\\
&= -\frac{a^2}{F(ia)}\frac{a}{x^2+a^2}. 
\end{align*}

Hence
\begin{align}\label{M+ta-eq4}
-\frac{1}{a} \int_0^\infty h_a''(w) e^{-xw} dw = \frac{a^2}{F(ia)} \frac{1}{x^2+a^2} - \frac{1}{a} \int_0^\infty h_a''(-w) e^{-xw} dw.
\end{align}

Since $h_a''(-w)\ge 0$ for $w\geq 0$, we have from \eqref{ha(0)}
\begin{align}\label{M+ta-eq5}
\int_0^\infty h_a''(-w) e^{-xw} dw \le \int_0^\infty h_a''(-w) dw =\frac{g'(0)}{a} +\frac{a}{2F(ia)}.
\end{align}

We multiply \eqref{M+ta-eq5} by $-a^{-1}$ and insert the resulting inequality into \eqref{M+ta-eq4} to get
\[
-\frac{1}{a} \int_0^\infty h_a''(w) e^{-xw} dw  \ge \frac{a^2}{F(ia)} \frac{1}{x^2+a^2} - \frac{g'(0)}{a^2} -\frac{1}{2F(ia)}.
\]

Inserting this into \eqref{M+ta-eq3} gives
\[
\frac{M^+(F,a,x)-t_a(x)}{F(x)} \ge \frac{1}{F(ia)}\frac{1}{x^2}\left(\frac{1}{(x/a)^2+1} - 1\right),
\]
which is nonnegative since $F(ia)<0$. This proves \eqref{M^+Sign} for $x>0$.
\end{proof}
 
 \begin{theorem}\label{M^-thm} Let $a >0$. Let $F\in LP$ be even with a double zero at the origin and at least one positive zero $\tau$. Assume that $F$ is positive in $(0,\tau)$, and define $g = g_{\tau/2}$ by \eqref{g_cDef} and $h_a$ by \eqref{h-def}. Then
\begin{align}\label{M^-Sign}
F(x)\left\{M^-(F,a,x) - t_a(x)\right\}\leq 0
\end{align}
holds for all real $x$. 
\end{theorem}

\begin{proof}  We note that the integral representations for $M^+$ are valid even if $F(ia)$ is not negative. From the definition of $M^-$ and  \eqref{M+ta-eq1} we obtain for $x<0$ the representation
\begin{align}\label{M-ta-eq2}
M^-(F,a,x) - t_a(x) = \frac{F(x)}{a} \int_{-\infty}^0 (h_a(w) - h_a(0)) e^{-xw} dw,
\end{align}
and since $h_a'(w)\ge 0$ for real $w$, it follows that $h_a(w)- h_a(0)\le 0$ for $w\le 0$ which shows \eqref{M^-Sign} for $x<0$. Analogously, for $x>0$ 
\begin{align}\label{M-ta-eq3}
M^-(F,a,x) - t_a(x) = -\frac{F(x)}{a} \int_{-\infty}^0 (h_a(w) - h_a(0)) e^{-xw} dw,
\end{align}
which gives \eqref{M^-Sign} in this range.
\end{proof}

\begin{proposition}  The functions $M^+$ and $M^-$ from Theorems  \ref{M^+thm} and  \ref{M^-thm} satisfy
 \begin{align}\label{M-integrability-estimate}
 |M^\pm(F,a,x) - t_a(x)| = \mc{O}\left( \frac{ F(x)}{ 1+x^2}\right)
 \end{align}
 for all real $x$.
 \end{proposition}
 
 \begin{proof}  Inequalities \eqref{M+ta-eq1} and \eqref{M+ta-eq3} yield \eqref{M-integrability-estimate} for $M^+$, while \eqref{M-ta-eq2} and \eqref{M-ta-eq3} imply \eqref{M-integrability-estimate} for $M^-$.
 \end{proof}

%%%
%%%           Extremal functions for t_a in de Branges spaces
%%%

\section{Extremal functions for $t_a$ in de Branges spaces}\label{DBS}

In this section we prove Theorem \ref{T1}.  Recall that $E$ is an HB function that satisfies  (I) $E$ is of bounded type in $\C^+$ with mean type $\tau$, (II) $E$ has no  real zeros, (III) $E^*(z) = E(-z)$ for all $z$, (IV) $B \notin \mc{H}(E)$. These conditions imply that $E$ has certain properties that we collect in the following lemma. We define the positive Borel measure $\mu_E$ by
\begin{align}\label{DBS-measure}
\mu_E(A) = \int_A \frac{dx}{|E(x)|^2}.
\end{align}

\begin{lemma}\label{E-properties} If $E$ is an HB function satisfying (I) - (IV), then 
\begin{enumerate}
\item $E$ has exponential type $\tau$,
\item Every nonnegative $F\in\mc{A}_1(2\tau, \mu_E)$ can be factored as $F = UU^*$ with $U\in\mc{H}(E)$,
\item $A = (1/2) (E+ E^*) $ is even and $B = (i/2)(E-E^*)$ is odd,
\item For every $U\in\mc{H}(E)$ the identity
\begin{align}\label{qu-HE}
\int_{-\infty}^\infty \left| \frac{U(x)}{E(x)}  \right|^2 dx = \sum_{\xi \in \mc{T}_B} \frac{|U(\xi)|^2}{K(\xi,\xi)}
\end{align}
is valid. (Recall that $\mc{T}_B$ is the set of zeros of $B$.)
\end{enumerate}
\end{lemma}

\begin{proof} Since $E$ is of bounded type with mean type $\tau$ in $\C^+$, it follows from \eqref{HB} that $E^*$ is of bounded type with mean type $\le \tau$ in $\C^+$. It follows from Krein's theorem \cite[Theorems 6.17 and 6.18]{RR} that $E$ has exponential type $\tau$. The second property follows from \cite[Appendix V]{Lev} together with the observation that $F\in\mc{A}_1(2\tau, \mu_E)$ implies 
\[
\int_{-\infty}^\infty \frac{\log^+|F(x)|}{1+x^2} dx<\infty.
\]

This follows from Jensen's inequality, see, e.g., the proof of \cite[Lemma 12]{HV}. The third property is evident, and \eqref{qu-HE} is \cite[Theorem 22]{B}.
\end{proof}

\begin{proof}[Proof of Theorem \ref{T1}] Assume that $E$ satisfies (I) - (IV).  Inequality \eqref{HB} implies that $A$ and $B$ have only real zeros, and since $E$ has no real zeros, it follows that $A$ and $B$ have no common zeros. By \cite[Problem 48]{B} there exists a continuous increasing function $\varphi$ such that $E(x) \exp(i\varphi(x))$ is real valued for all real $x$. The zeros of $A$ are the values $x$ where $\varphi(x) = \frac{\pi}{2}+k\pi$ for some $k\in \Z$, and the zeros of $B$ are the values $x$ where $\varphi(x) = k\pi$ for some $k\in\Z$. It follows that the zeros interlace. It follows from \cite[Problem 47]{B} applied to $f = A/B$ that the zeros of $B$ are simple (see also the proof of Theorem 22 in \cite{B}). Similarly, the zeros of $A$ are simple. 

 Since $E$ has exponential type $\tau$, it follows that $A$ and $B$ are also entire functions of exponential type $\tau$. Hence they, and their squares, are Laguerre-P\'olya entire functions. Evidently, $B^2$ is an even LP function that has a double zero at the origin, and the results of Section \ref{interpolation} are applicable.

We define the entire functions $T_a^+$ and $T_a^-$ by
\begin{align*}
T^+_a(z) &= M^+(B^2, a, z)\\
T_a^-(z) &= M^-(B^2,a,z)
\end{align*}
with $M^-$ and $M^+$ as in \eqref{mDef} and \eqref{M^+Def}.
\medskip

We note that $E^*(z) = E(-z)$ implies in particular that $\overline{B(ix)} = -B(ix)$ for real $x$, i.e., $\Re B(ix) =0$ for real $x$. It follows then that
\[
\big(B(ix)\big)^2 <0
\]
for every real $x$, hence $F=B^2$ satisfies the assumptions of Theorem \ref{M^+thm}.  Since $B^2\ge 0$ on $\R$, inequality \eqref{M^+Sign} implies
\[
T^+_a(x) \ge t_a(x)
\] 
for all real $x$, and \eqref{interpolationPoints} implies that 
\[
T^+_a(\xi) = t_a(\xi)
\]
for all $\xi$ with $B(\xi)=0$. Since $B^2/E^2$ is bounded on $\R$, it follows from \eqref{M-integrability-estimate} that
\[
\int_{-\infty}^\infty \frac{T_a^+(x) - t_a(x)}{|E(x)|^2} dx<\infty.
\]

A similar argument gives the same statement for $t_a - T_a^-$. Since $T_a^- \le t_a\le T_a^+$ we obtain that $|T_a^+ - T_a^-|$  is integrable with respect to $\mu_E = |E(x)|^{-2}dx$.  It follows from Lemma \ref{E-properties} that there exists $U\in\mc{H}(E)$ such that
\begin{align}\label{T-decomposition}
T_a^+ - T_a^- &= UU^*.
\end{align}

We prove next the optimality of $T_a^+$. Let $F$ be a function of type $2\tau$ with $F\ge t_a$ on $\R$. We may assume that
\[
\int_{-\infty}^\infty \frac{ F(x)  - t_a(x)}{|E(x)|^2} dx <\infty,
\]
(since otherwise there is nothing to show). The inequality $T_a^-\le t_a\le F$ gives
\[
|F(x) - T_a^-(x)| \le  (F(x) - t_a(x)) + (t_a(x) - T_a^-(x))
\]
and hence $F - T_a^-$ is an entire function of exponential type $2\tau$ that is integrable with respect to $\mu_E$. Evidently, $F - T_a^- \ge F - t_a \ge 0$. Applying Lemma \ref{E-properties} again implies that there exists $V\in\mc{H}(E)$ such that
\begin{align}\label{F-decomposition}
F - T_a^- &= VV^*,
\end{align}

It follows from \eqref{T-decomposition} and \eqref{F-decomposition} that
\[
F- T_a^+  =  VV^* - UU^*.
\]

An application of Lemma \ref{E-properties} (4) to $U$ and $V$ together with $T_a^+(\xi) = t_a(\xi)$ for all $\xi\in\R$ with $B(\xi)=0$ implies
\[
\int_{-\infty}^\infty \frac{F(x) - T_a^+(x)}{|E(x)|^2} dx = \sum_{B(\xi)=0} \frac{F(\xi) - t_a(\xi)}{K(\xi,\xi)}\ge 0,
\]
hence $T_a^+$ is extremal.

An analogous calculation (which we omit) shows that $T_a^-$ is an extremal minorant. It remains to prove that 
\[
\int_{-\infty}^\infty (T_a^+(x) - T_a^-(x)) \frac{dx}{|E(x)|^2} = \frac{1}{a^2 K(0,0)}.
\] 

It follows from \eqref{T-decomposition} and Lemma \ref{E-properties} (4) that
\[
\int_{-\infty}^\infty (T_a^+(x) - T_a^-(x)) \frac{dx}{|E(x)|^2} = \sum_{B(\xi)=0} \frac{T_a^+(\xi) - T_a^-(\xi)}{K(\xi,\xi)}.
\]

The only non-zero summand is the term for $\xi=0$. Since $T_a^+(0) - T_a^-(0) = t_a(0) - t_a(0-) =1/a^2$, the proof is complete.
\end{proof}

%%%
%%%           De Branges space for the vanishing condition
%%%

\section{De Branges space and optimal functions for the vanishing condition}\label{section-vanishing-space}

Define the Borel measure $\mu_a$ by
\[
\mu_a(B) = \int_B (x^2+a^2) dx.
\]

Let $a>0$ and recall that $E_a$ is given by
\[
E_a(z) = \sqrt{\frac{2}{\sinh(2\pi a)}}\frac{\sin\pi(z+ia)}{(z+ia)}.
\]

We prove that $E_a$ is a Hermite-Biehler function whose associated de Branges space is isometrically equal to $\mc{A}_2(\pi,\mu_a)$.

\begin{theorem}\label{isometry} Let $a>0$. The function $E_a$ satisfies \eqref{HB} and properties (I) - (IV). Moreover, the space $\mc{A}_2(\pi,\mu_a)$ is isometrically equal to $\mc{H}(E_a)$.
\end{theorem}
 
 \begin{proof}  Since $z\mapsto \sin\pi z$ is LP and hence of P\'{o}lya class, we have that $z \mapsto \sin(\pi(z + i a))$ is also of P\'{o}lya class. By \cite[Section 7, Lemma 1]{B} it follows that $E_a$ is of P\'{o}lya class. This implies
\[
|E_a(z)|\ge |E_a^*(z)|
\]
for all $z$ with $\Im z>0$. Since $E_a$ has no zeros in the upper half plane, the function $E_a^*/E_a$ is analytic in the upper half plane and has modulus bounded by $1$. Since this quotient is not constant, the modulus is never equal to $1$ by the maximum principle, hence $E_a$ satisfies \eqref{HB}. 

It can be checked directly that $E_a$ has bounded type $\pi$ (or apply the reverse direction of Krein's theorem). Evidently $E_a$ has no real zeros and $E_a^*(z) = E_a(-z)$ for all $z$. A direct calculation gives
\begin{align*}
A_a(z) &= \sqrt{\frac{2}{\sinh(2\pi a)}} \frac{z\cosh(\pi a)\sin(\pi z) +a \sinh (\pi a)\cos(\pi z)}{z^2+a^2},\\
B_a(z) &= \sqrt{\frac{2}{\sinh(2\pi a)}} \frac{a\cosh(\pi a)\sin(\pi z) -z \sinh (\pi a)\cos(\pi z)}{z^2+a^2},
\end{align*}
and in particular $B_a\notin \mc{H}(E_a)$.  Hence $E_a$ satisfies (I) - (IV).  

Taking limits in \eqref{K-from-AB} leads to the representation
\begin{align}\label{RK-real}
K_a(x,x) = \frac{\pi(a^2+x^2) -a\coth(2\pi a) + a\cos(2\pi x) \csch(2\pi a)}{\pi (a^2+x^2)^2}
\end{align}
for all real $x$. Recall that $\mu_{E_a}(A) = \int_A |E_a(x)|^{-2} dx$. It is straightforward to check that $L^2(\R,\mu_a)$ and $L^2(\R, \mu_{E_a})$ are equal as sets with equivalent norms. It follows that $\mc{A}_2(\pi,\mu_a)$ and $H(E_a)$ are equal as sets and have equivalent norms. The main statement to prove is the fact that the two norms are equal on the smaller spaces.

We note first that
\begin{align}\label{Sinc-eq1}
\frac{(z+i a)(z-ia)}{\sin(\pi(z+ia))\sin(\pi(z-ia))} = \frac{2(z^2+a^2)}{\cosh(2\pi a) - \cos(2\pi z)}
\end{align}
holds, in particular, the right hand side is $1$-periodic after division by $z^2+a^2$. Furthermore,
\[
\int_0^1 \frac{1}{\cosh(2\pi a) - \cos(2\pi x)} dx = \frac{1 }{\sinh(2\pi a)}.
\]

This means that $p_a$ defined by  $p_a(x) =  \sinh(2\pi a) (\cosh(2\pi a) - \cos(2\pi x))^{-1} -1$ is $1$-periodic and has mean value zero. Since $a>0$, this function is infinitely  differentiable on the real line. It follows that the Fourier series of $p_a$ converges absolutely and uniformly, and that it represents the function, i.e., there exists a sequence $a_n$ so that
\begin{align}\label{rep-pa}
p_a(x) = \sum_{n\neq 0} a_n e^{2\pi i n x}
\end{align}
for all real $x$. 

  Let $H\in L^1(\R)$ be an entire function of exponential type $2\pi$. Since $H\in L^2(\R)$  by \cite[Theorem 6.7.1]{Bo}, the Paley-Wiener theorem \cite[Theorem  4.1]{SW} implies that the Fourier transform $\widehat{H}$ defined by
  \[
  \widehat{H}(t) = \int_{-\infty}^\infty e^{-2\pi i x t} H(x) dx
  \]
satisfies $\widehat{H}(t) = 0$ for $|t|>1$. Since $H\in L^1(\R)$ it follows that $\widehat{H}$ is continuous, hence $\widehat{H}(t) = 0$ for $|t|\ge 1$. This implies
\[
\int_{-\infty}^\infty H(x) \sum_{\substack{|n|\le N \\ n\neq 0}} a_n e^{2\pi i n x}   dx = \sum_{\substack{ |n|\le N \\ n\neq 0} } a_n \widehat{H}(-n) = 0. 
\]

Since the partial sums of the series in  \eqref{rep-pa} converge uniformly,  we obtain with an application of Lebesgue dominated convergence that
\begin{align}\label{zero-mean}
\int_{-\infty}^\infty H(x) \left( \frac{\sinh(2\pi a)}{\cosh(2\pi a) - \cos(2\pi x)} -1\right) dx = 0.
 \end{align}
  
Let $F,G\in \mc{A}_2(\pi,\mu_a)$ and define $H$ by $H(z) = F(z)G^*(z) (z^2+a^2)$. It follows from \eqref{Sinc-eq1}  that
\begin{align*}
\langle F, G \rangle_{\mc{H}(E_a)} - \langle F, G \rangle_{L^2(\R, \mu_a)} &= \int_{-\infty}^\infty F(x) G^*(x) \{|E_a(x)|^{-2}-(x^2+a^2)\} dx\\
&= \int_{-\infty}^\infty H(x) \left( \frac{\sinh(2\pi a)}{\cosh(2\pi a) - \cos(2\pi x)} -1\right) dx,
\end{align*}
and since $H$ is a Lebesgue integrable entire function of exponential type $2\pi$, it  follows from \eqref{zero-mean} that
\[
\langle F, G \rangle_{\mc{H}(E_a)} = \langle F, G \rangle_{L^2(\R, \mu_a)} 
\]
as claimed.
\end{proof}

\begin{proof}[Sketch of an alternate proof.]  Define for $a>0$ the meromorphic function $W_a$ by
\[
W_a(z) = -e^{-2\pi  a  } \frac{a+iz}{a-iz}
\]
and note that $W_a$ is analytic and has modulus $\le 1$ in the upper half plane. The identity
\[
\frac{E_a(z)+ E_a^*(z) W_a(z)}{E_a(z) - E_a^*(z) W_a(z)} = \coth(2\pi a) - e^{2\pi i z} \csch(2\pi a) 
\]
is valid for $z\in\C$, and for real $x$ and $y>0$
\[
\frac{y}{\pi} \int_{-\infty}^\infty \frac{(t^2+a^2) |E_a(t)|^2 dt}{(x-t)^2+y^2} = \coth(2\pi a) -  e^{-2\pi y}\cos(2\pi x) \csch(2\pi a)
\]
holds. Theorem V.A of \cite{B2} with $d\mu(t) = (t^2+a^2) |E_a(t)|^2 dt$ implies 
\[
\int_{-\infty}^\infty |f(x)|^2 (x^2+a^2) dx = \int_{-\infty}^\infty \left|\frac{f(x)}{E_a(x)}\right|^2 dx
\]
for every $f\in\mc{H}(E_a)$.
\end{proof}

\begin{proof}[Proof of Theorem \ref{vanishing}]  
We define $S_a^+$ and $S_a^-$ by
\begin{align*}
S_a^+(z) &= M^+(B_a^2, a, z)(z^2+a^2),\\
S_a^-(z) &= M^-(B_a^2, a, z)(z^2+a^2).
\end{align*}

Since $B_a$ is odd, we have $B_a(0) =0$. Since $E_a$ is Hermite-Biehler and is of bounded type in the upper half plane, it follows that $B_a$ is of bounded type in the upper half plane. Since $B_a$ has only real zeros, \cite[Problem 34]{B} implies that $B_a$ is of P\'{o}lya class, and \cite[Theorem 7]{B} implies that $B_a\in LP$. Hence $F = B_a^2$ satisfies the assumptions of Proposition \ref{continuationProp}. Furthermore, it can be checked directly that $B_a^2(ia) <0$.  For the remainder of the proof we set $M^\pm(z) = M^\pm(B_a^2,a,z)$. Since $B_a^2\in \mc{A}(2\pi)$ we obtain that $M^+,M^-\in \mc{A}(2\pi)$. Theorems \ref{M^+thm} and \ref{M^-thm} imply
\[
M^-(x)\le t_a(x) \le M^+(x)
\]
for all real $x$. The proof of Theorem \ref{T1} shows that $M^+ = T_a^+$ and $M^- = T_a^-$ with respect to the measure $d\mu_{E_a}$. It follows from \eqref{opt-value-intro} that
\begin{align}\label{opt-eq-1}
\int_{-\infty}^\infty (M^+(x) - M^-(x)) \frac{dx}{|E_a(x)|^2} = \frac{1}{ a^2 K_a(0,0)}=\frac{\pi a}{\pi a - \tanh(\pi a)}.
\end{align}

By definition of $S_a^\pm$ we have
\begin{align}\label{opt-eq-2}
\int_{-\infty}^\infty (S_a^+(z) &- S_a^-(x))dx= \int_{-\infty}^\infty (M^+(x)- M^-(x))  (x^2+a^2) dx.
\end{align}

Since $M^+-M^-\in \mc{A}_1(2\pi, \mu_a)$, it follows that $M^+ - M^- \in \mc{A}_1(2\pi, \mu_{E_a})$. Lemma \ref{E-properties} (2) implies that $M^+ - M^- = UU^*$ with $U\in \mc{H}(E_a)$. Theorem \ref{isometry}, \eqref{opt-eq-1}, and \eqref{opt-eq-2} imply 
\[
\int_{-\infty}^\infty (S_a^+(x) - S_a^-(x))dx = \frac{\pi a}{\pi a - \tanh(\pi a)},
\]
which gives the case of equality in \eqref{optimal-vanishing-intro}.

Let now $S,T\in\mc{A}(2\pi)$ such that $S(ia) = T(ia) =0$ and $S(x) \le t_a(x)  \le T(x)$ on the real line. We may assume that $S- M^-$ and $T - M^+$ are integrable with respect to $(x^2+a^2) dx$. Since $S$ and $T$ are real entire, it follows that $S(-ia) = T(-ia) =0$, hence 
\[
S(z) = (z^2+a^2) \sigma(z)
\]
and
\[
T(z) = (z^2+a^2) \tau(z)
\]
where $\sigma,\tau$ are entire and have exponential type $2\pi$. Furthermore, $\sigma - t_a$ and $\tau-t_a$ are integrable and
\[
\sigma(x) \le t_a(x) \le \tau(x)
\]
for all real $x$. It follows from Theorem \ref{T1} that
\[
\int_{-\infty}^\infty (\sigma(x) -\tau(x)) (x^2+a^2) dx \ge \frac{1}{a^2 K_a(0,0)},
\]
which is \eqref{optimal-vanishing-intro}.
\end{proof}

\end{document}